\documentclass{amsart}
\usepackage{amssymb}
\usepackage{latexsym}
\usepackage{amsmath}
\usepackage{mathrsfs}
\usepackage[X2,T1]{fontenc}
\usepackage[applemac]{inputenc}
\usepackage{cite}
\usepackage{calc}                   

\newtheorem{theorem}{Theorem}[section]
\newtheorem{lemma}[theorem]{Lemma}
\newtheorem{proposition}[theorem]{Proposition}
\theoremstyle{definition}

\def\R{\mathbb R}

\def\N{\mathbb N}

\def\Q{\mathbb Q}

\def\1{\mathbb 1}
\def\cS{\mathscr{S}}

\theoremstyle{remark}
\newtheorem{remark}[theorem]{Remark}

\numberwithin{equation}{section}

\begin{document}

\title{ANISOTROPIC SHUBIN OPERATORS AND EIGENFUNCTION EXPANSIONS IN GELFAND-SHILOV SPACES}

\author{Marco Cappiello}
\address{Dipartimento di Matematica, Universit\`a di Torino, Via Carlo Alberto 10, 10123 Torino, Italy}
\email{marco.cappiello@unito.it}

\author{Todor Gramchev}
\address{Dipartimento di Matematica e Informatica, Universit\`a di Cagliari, Via Ospedale 72, 09124 Cagliari, Italy}

\author{Stevan Pilipovic}
\address{Institute of Mathematics, University of Novi Sad, trg. D. Obradovica 4, 21000 Novi Sad, Serbia}
\email{stevan.pilipovic@uns.dmi.ac.rs}
\author{Luigi Rodino}
\address{Dipartimento di Matematica, Universit\`a di Torino, Via Carlo Alberto 10, 10123 Torino, Italy}
\email{luigi.rodino@unito.it}





\begin{abstract} We derive new results on the characterization of Gelfand--Shilov spaces $\mathcal{S}^\mu_\nu (\R^n)$, $\mu,\nu >0$,
$\mu+\nu \geq 1$ by Gevrey estimates of the $L^2$ norms of iterates of $(m,k)$ anisotropic globally elliptic Shubin (or $\Gamma$) type operators, $(-\Delta)^{m/2} +| x |^k$ with $m,k\in 2\N$ 
 being a model operator,   and on the decay of the Fourier coefficients in the related eigenfunction expansions. Similar results are obtained for the spaces $\Sigma^\mu_\nu (\R^n)$, $\mu,\nu >0$, $\mu+\nu > 1$, cf. \eqref{GSdef}. In contrast to the symmetric case $\mu = \nu$ and $k=m$ (classical Shubin operators) we encounter resonance type phenomena involving the ratio $\kappa:=\mu/\nu$; namely we obtain a characterization of $\mathcal{S}^\mu_\nu(\R^n)$ and $\Sigma^\mu_\nu(\R^n)$ in the case $\mu=kt/(k+m), \nu= mt/(k+m), t \geq 1$, that is, when $\kappa=k/m \in \Q$. 
\end{abstract}

\maketitle

\section{Introduction and statement of the results}

The main goal of the paper is to prove results on the characterization of the  non-symmetric ($\mu \neq \nu)$ Gelfand--Shilov spaces $\mathcal{S}^\mu_\nu (\R^n)$, $\mu,\nu >0$,
$\mu+\nu \geq 1$ by Gevrey estimates of the $L^2$ norms of the iterates  $P^\ell u $, $\ell =1,2,\ldots, u\in \cS(\R^n),$ of positive anisotropic globally elliptic Shubin differential operators $P$ of the type $(m,k)$, $m,k$ being  even natural numbers, and on the decay of the Fourier coefficients $u_j$, $j\in \N$, in the eigenfunction expansions $u = \sum_{j=1}^\infty u_j \varphi_j$, where $\{ \varphi_j\}_{j=1}^\infty$ stands for an orthonormal basis of eigenfunctions associated to the  operator $P$. 
The $(m,k)$ Shubin elliptic differential operators are modelled by  
\begin{eqnarray}
{\mathcal H}^{m,k}_n:= (-\Delta)^{m/2} + |x|^{k}, \ \ \ |x| = \sqrt{ x_1^2 +\ldots +x_n^2}, \,  k, m\in 2\N.
\label{anismod1}
\end{eqnarray} 
We recall that for $\mu >0, \nu>0,$ the inductive  (respectively, projective) Gelfand-Shilov
classes $\mathcal{S}^\mu_\nu({\mathbb R}^n),  \; \mu+\nu\geq 1$ (respectively, $\Sigma^\mu_\nu (\R^n), \; \mu +\nu >1$),   
are defined as the set of all $u\in \cS(\R^n)$ for which there exist $A>0, C>0$ (respectively, for every $A>0$ there exists $C>0$) such that  
\begin{equation} \label{GSdef}
|x^\beta\partial_x^\alpha u(x)|\leq
CA^{|\alpha|+|\beta|}(\alpha!)^\mu(\beta!)^\nu, \;\; \alpha, \beta \in {\mathbb N}^n,
\end{equation}
 see  \cite{GeShi, A, CCK, CPRT, Mi}  and \cite[Chapter 6]{NR}.
These spaces have recently gained a wide importance in view of the fact that they represent a suitable functional setting both for microlocal analysis and PDE and for Fourier and time-frequency analysis \cite{AC, BG, CGR1, CGR2, CGR3, CNR, CaRo, CT, GZ2, T2}.

Concerning the investigation in the present paper, we can cite different sources of motivations. First, we recall the fundamental work of Seeley \cite{S2}  on eigenfunction expansions of real analytic functions on compact manifolds (see also the recent paper of 
Dasgupta and Ruzhansky \cite{DaRu1}, extending the result of \cite{S2} for all Gevrey spaces $G^\sigma$, $\sigma >1$, on compact Lie groups). 
Secondly, we mention the work \cite{gpr1} on the characterization of symmetric Gelfand-Shilov spaces 
$\mathcal{S}^\mu_\mu (\R^n)$ by means of estimates of iterates and the decay of the Fourier coefficients in the eigenfunction expansions associated to globally elliptic (or $\Gamma$ elliptic) differential operator. We also refer to \cite{vv}, where general Gevrey sequences $M_p$ are used. 
Finally, we mention as additional motivation the results on hypoellipticity in $\mathcal{S}^\mu_\nu (\R^n)$ for elliptic operators of the type $\mathcal{H}^{m,k}_n$ for $\mu \geq k/(m+k)$, $\nu \ge m/(m+k)$, $k,m$ being even natural numbers, cf. \cite{CGR2} (see also the older work \cite{CGR1}).

Before stating our main results we need some preliminaries.

As counterpart of an elliptic operator in a compact manifold, we consider in $\R^n$ the decay of the Fourier coefficients in the eigenfunction expansions associated to $\mathcal{H}^{m,k}_n $. In contrast to the symmetric case $\mu = \nu$ and $k=m$ (classical Shubin operators) we encounter new resonance type phenomena involving $\kappa:=\mu/\nu$, namely we can characterize the spaces $\mathcal{S}^\mu_\nu(\R^n)$, $\mu +\nu \geq1$ (respectively $\Sigma^\mu_\nu(\R^n)$, $\mu +\nu >1$) by iterates and eigenfunction expansions defined by $\mathcal{H}^{m,k}_n$  iff $\kappa$ is rational number, $\kappa=k/m$.

Our basic example of operator will be the anisotropic quantum harmonic
oscillator appearing in Quantum Mechanics 
\begin{eqnarray}\label{eq1.1}
{\mathcal H}^{2,k}_n=-\triangle+| x|^k, \qquad k \in 2\N,
\end{eqnarray}
with recovering for $k=2$ the standard harmonic oscillator 
whose eigenfunctions are the Hermite functions
\begin{eqnarray}\label{eq1.2}
h_{\alpha}(x)=H_\alpha(x)e^{-|x|^2/2}, \;\;\;
\alpha=(\alpha_1,...,\alpha_n)\in {\mathbb N}^n,
\end{eqnarray}
where $H_\alpha(x)$ is the $\alpha$-th Hermite polynomial. See for example \cite{lang, pil88, reedsimon}  for related Hermite
expansions as well as  \cite{gpr2, Wpar} for connections
with a degenerate harmonic oscillator.

Here we shall consider a more general class of operators with
polynomial coefficients in ${\mathbb R}^n$, namely $(m,k)$ anisotropic operators:
\begin{eqnarray}\label{eq1.3}
P=\sum\limits_{\frac{|\alpha|}{m} +\frac{|\beta|}{k} \leq 1}c_{\alpha\beta}x^\beta
D_x^\alpha, \;\;\; D^\alpha=(-i)^{|\alpha|}\partial_x^\alpha.
\end{eqnarray}
Set 
\begin{eqnarray}\label{w1}
\Lambda_{m,k}(x,\xi)  & = & (1 + | x|^{2k} + | \xi| ^{2m})^{1/2}, \quad (x,\xi)\in \R^{2n}, \ \textrm{$m,k\in 2\N$.}
\end{eqnarray} 

The global ellipticity for $P$ in (\ref{eq1.3}) is defined by
imposing
\begin{eqnarray}\label{eq1.4}
\sum\limits_{\frac{|\alpha|}{m}+\frac{|\beta|}{k}=1}c_{\alpha\beta}x^\beta\xi^\alpha
\neq 0 \;\;\; \mbox{for} \;\; (x,\xi)\neq (0,0).
\end{eqnarray}
or equivalently, there exist $C_1>0, C_2>0, R>0$ such that 
\begin{eqnarray}\label{eq1.4a}
C_2 \leq \frac{|p(x,\xi)|}{\Lambda_{m,k}(x,\xi)}\leq C_1, \quad |(x,\xi)|\geq R.
\end{eqnarray}
Under the assumption \eqref{eq1.4} (or \eqref{eq1.4a}), the following estimate holds for every $u \in \cS(\R^n)$:
\begin{equation} \label{ellipticestimate}
\sum_{\frac{|\alpha|}{m} +\frac{|\beta|}{k} \leq 1} \| x^\beta
D_x^\alpha u \|_{L^2} \leq C( \|Pu\|_{L^2} + \|u \|_{L^2}),
\end{equation}
cf. \cite{BBR}.

 For
these operators, the counterpart of the standard Sobolev spaces
are the spaces $
Q^{s}_{m,k}({\mathbb R}^n), s \in \R,$ defined, for example,  by requiring that
\begin{eqnarray}\label{eq1.5a}
\left\|\Lambda(x,D)^s u\right\|_{L^2}<\infty,
\end{eqnarray}
where
\begin{eqnarray}\label{sublin1}
\Lambda(x,\xi) = (1 + | x|^{2k}+ | \xi|^{2m})^{1/2\max \{k, m \}}, \quad k,m\in 2\N.
\end{eqnarray}

 Under the
global ellipticity assumption (\ref{eq1.4}), $$P:Q^{s}_{m,k}({\mathbb R}^n)
\to L^2({\mathbb R}^n),\; s = \max\{k,m\},$$ is a Fredholm operator. The
finite-dimensional null-space $\textrm{Ker}\, P$ is given by functions in the
Schwartz space $\cS({\mathbb R}^n)$.

We assume, as in \cite{gpr1},  that $P$ is a positive anisotropic elliptic operator, which implies that $k$ and $m$ are even numbers.
This guarantees the existence
of an orthonormal  basis of eigenfunctions $\varphi_j$, $j\in \N$, with eigenvalues $\lambda_j$, $\lim\limits_{j\to \infty }\lambda_j = +\infty$ (see \cite{shubin}). Moreover we have that
\begin{equation} \label{aseigen}
\lambda_j \sim C j^{\frac{mk}{n(m+k)}} \qquad \textit{as} \quad j \to +\infty.
\end{equation}
for some $C>0$, cf. \cite{BBR, shubin}.
Hence, given $u \in L^2({\mathbb R}^n)$, or $u \in
\cS '({\mathbb R}^n)$, we can expand
\begin{eqnarray}\label{eq1.7}
u=\sum\limits_{j=1}^\infty u_j \varphi_j
\end{eqnarray}
where the Fourier coefficients $u_j \in \mathbb C$ are defined by
\begin{eqnarray}\label{eq1.8}
u_j=(u,u_j)_{L^2}, \;\; j=1,2,\ldots
\end{eqnarray}
with convergence in $L^2({\mathbb R}^n)$ or $\cS '({\mathbb R}^n)$ for (\ref{eq1.7}).

By the hypoellipticity results of \cite{CGR2} the eigenfunctions $\varphi_j$ belong to $\mathcal{S}^{k/(m+k)}_{m/(m+k)}(\R^n)$.

We first state an assertion on the characterization of the anisotropic Sobolev spaces $Q^s_{m,k} (\R^n)$ and the Schwartz class $\cS(\R^n)$.

\begin{theorem}\label{t1.1}
Suppose that $P$ is $(m,k)$-globally elliptic  cf. \eqref{eq1.3}, \eqref{eq1.4}, and positive. Then:
\begin{itemize}
\item [(i)] $u \in Q^s_{m,k}({\mathbb R}^n) \Longleftrightarrow \sum\limits_{j=1}^\infty 
|u_j|^2 \lambda_j^{s/\max \{ m,k \} }< \infty $, $s \in {\mathbb N}$. 
\item [(ii)] $u \in \cS ({\mathbb R}^n)  \Longleftrightarrow  |u_j|= O(\lambda_j^{-s}), \, j\to \infty 
\Longleftrightarrow  |u_j|= O(j^{-s}), \, j\to \infty$ for all $s\in \N$.
\end{itemize}
\end{theorem}

Let us now come to the characterization of the spaces $\mathcal{S}^\mu_\nu(\R^n)$ and $\Sigma^\mu_\nu(\R^n)$ in the case $\kappa:=\mu/\nu \in \Q.$ We may link $\mu, \nu$ with an operator of the form \eqref{eq1.3} for a suitable choice of $k$ and $m$. In fact, observe first that we may write $\mu=t\mu_o, \nu=t \nu_o$ for some $t >0$ with $\mu_o=\kappa/(1+\kappa), \nu_0=1/(1+\kappa)$ so that $\mu_o+\nu_o = 1$. If $\mu+\nu \geq 1$ we have $t \geq 1$, if $\mu+\nu >1$ then $t >1$. On the other hand, for any given $\mu_o \in \Q$ we may write $\mu_o =k/(k+m)$ for two positive integers $k$ and $m$, and consequently $\nu_o=1-\mu_o=m/(k+m).$ Multiples of $k$ and $m$ work as well, in particular we may assume $k$ and $m$ to be even natural numbers so that the symbol of $\Lambda_{m,k}$ in (\ref{w1})
is a smooth function which is necessary for the proof of the hypoellipticity result of \cite{CGR2}. So we have $$\mu=\frac{kt}{k+m}, \quad \nu = \frac{mt}{k+m}.$$
For given even integers $k$ and $m$, an example of globally elliptic positive operator is given by \eqref{anismod1}. \\

The first main result of the paper characterizes the Gelfand-Shilov spaces in terms of estimates of the iterates of $P$ and reads as follows.

\begin{theorem}\label{t1}
Let $P$ be an operator of the form \eqref{eq1.3} for some integers $k \geq 1, m \geq 1$, be globally elliptic, namely satisfy \eqref{eq1.4} and let $u \in \cS(\R^n)$. Then $u \in \mathcal{S}^{\frac{kt}{k+m}}_{\frac{mt}{k+m}}(\R^n), t \geq 1$ (respectively $u \in \Sigma^{\frac{kt}{k+m}}_{\frac{mt}{k+m}}(\R^n), t > 1$) if and only if there exist $C>0, R>0$ (respectively for every $C>0$ there exists $R>0$) such that:
\begin{equation} \label{iteratesestimate}
\|P^M u\|_{L^2}\leq RC^{M} (M!)^{\frac{kmt}{k+m}}
\end{equation}
for every integer $M \geq 1$.
\end{theorem}  

\begin{remark}
Theorem \ref{t1} suggests the possibility of considering new function spaces defined by the estimates \eqref{iteratesestimate}
also for $0< t <1$ (respectively $0<t \leq 1$). Corresponding Gelfand-Shilov classes are empty in that case as well known from \cite{GeShi} and the equivalence in Theorem \ref{t1} fails. Nevertheless such definition in terms of \eqref{iteratesestimate} deserves interest, cf. also \cite{CST, T1}.
\end{remark}

Using Theorem \ref{t1} we can prove the following result.

\begin{theorem}\label{t2}
Let $P$ be a positive operator of the form \eqref{eq1.3} for some integers $k \geq 1, m \geq 1$, satisfying \eqref{eq1.4} and let $u \in \cS(\R^n)$.  Let the eigenvalues $\lambda_j$ and the Fourier coefficients $u_j$ be defined as before. The following conditions are equivalent:
\\
i) $u \in \mathcal{S}^{\frac{kt}{k+m}}_{\frac{mt}{k+m}}(\R^n), t \geq 1$ (respectively $u \in \Sigma^{\frac{kt}{k+m}}_{\frac{mt}{k+m}}(\R^n), t > 1$); \\
ii) there exists $\varepsilon>0$ such that (respectively for every $\varepsilon >0$) we have
\begin{equation}\label{aseigenexp1}
\sum_{j=1}^\infty |u_j|^2 e^{\epsilon \lambda_j^{\frac{k+m}{kmt}}} < \infty;
\end{equation}
iii) there exists $\varepsilon>0$ such that (respectively for every $\varepsilon >0$) we have
\begin{equation}\label{aseigenexp2}
\sup_{j\in \N} |u_j|^2 e^{\epsilon \lambda_j^{\frac{k+m}{kmt}}} < \infty.
\end{equation}
iv) there exists $\varepsilon>0$ such that (respectively for every $\varepsilon >0$) we have for some $C>0$:
$$|u_j| \leq C e^{-\varepsilon j^{\frac{1}{tn}}}, \qquad j \in \N.$$
\end{theorem}

The somewhat surprising fact that in $ iv)$  the estimates do not depend on the couple $(m,k)$, that is on $(\mu,\nu)$, may find intuitive explanation in the $\mathcal{S}^\mu_\nu$ regularity of the eigenfunctions $\varphi_j$, cf. \cite{CGR2}.

\section{Proof of the main results}
\textit{Proof of Theorem \ref{t1.1}.}
The proof of Theorem \ref{t1.1} is easy, by using the $r$-th power of $P, r \in \R$, that we may define as 
$$P^r u = \sum_{j=1}^\infty \lambda_j^r u_j \varphi_j,$$
and by observing that the norms $\|P^r u\|_{L^2}, r= s/\max\{k,m \}$ and $\| \Lambda(x,D)^s u\|_{L^2}$ are equivalent, see \cite{BBR, NR, shubin}. On the other hand, by Parseval identity 
$$\|P^r u\|_{L^2}^2 = \| \sum_{j=1}^\infty \lambda_j^r u_j \varphi_j\|_{L^2}^2 = \sum_{j=1}^\infty \lambda_j^{2r}|u_j|^2$$
and $i)$ follows. Since $\cS(\R^n) = \bigcap\limits_{s \in \N} Q_{m,k}^s(\R^n)$ we also obtain $ii)$. \qed
\\

The proof of Theorem \ref{t1} needs some preparation. We first define, for fixed $ r \geq 0 $ and $u \in L^2(\R^n)$:
\begin{equation}
|u|_r = \sum_{\frac{|\alpha|}{m}+\frac{|\beta|}{k}=r} \| x^{\beta}D^\alpha u \|_{L^2}
\end{equation}

First it is useful to characterize Gelfand-Shilov spaces in terms of the norms $|u|_r$ as follows.

\begin{proposition} \label{char}
Let $u \in L^2(\R^n).$ Then $u \in \mathcal{S}^{\frac{kt}{k+m}}_{\frac{mt}{k+m}}(\R^n), t \geq 1$ (respectively $u \in \Sigma^{\frac{kt}{k+m}}_{\frac{mt}{k+m}}(\R^n), t > 1$) if and only if there exist $C>0, R>0$ (respectively for every $C>0$ there exists $R>0$) such that
\begin{equation} \label{charest}
|u|_r \leq RC^{r} r^{\frac{kmrt}{k+m}}
\end{equation}
for every $ r>0$.
\end{proposition}

We have the following preliminary result.

\begin{lemma}
\label{AntiWick}
There exists a constant $C>0$ such that, for any given $p \in \N, (\alpha, \beta ) \in \N^{2n}$, with $|\alpha|/m+|\beta|/k =r, p<r<p+1,$ and for every $\varepsilon >0$, the following estimate holds true:
\begin{equation} \label{AW}
|u|_r \leq \varepsilon |u|_{p+1}+C \varepsilon^{-\frac{r-p}{p+1-r}}|u|_p + C^p (p+1)!^{\frac{km}{k+m}}|u|_0
\end{equation}
for all $u \in \cS(\R^n)$.
\end{lemma}

The proof follows the same lines as the proof of Proposition 2.1 in \cite{CR}, cf. also \cite{KN}, and it is omitted.

 \par

Next, fixed $\lambda >0, p \in \N$ and $ u \in L^2(\R^n)$, we set:
\begin{equation}
\sigma_p(u, \lambda) = \lambda^{-p} (p!)^{-\frac{kmt}{k+m}} |u|_p.
\end{equation}

\begin{lemma} \label{sigma}
For every $p \in \N$ and for  $\lambda >0$ sufficiently large, we have:
\begin{equation} \label{stimasigma}
\sigma_{p+1}( u, \lambda) \leq (p+1)^{-\frac{kmt}{k+m}} \sigma_p (Pu, \lambda) + \sum_{h=0}^{p} \sigma_h (u,\lambda)
\end{equation}
for every $u \in \cS(\R^n).$
\end{lemma}

\begin{proof} 
For $p=0$ the assertion is a direct consequence of \eqref{ellipticestimate} if $\lambda$ is large enough. 
Fix now $p \in \N, p \geq 1$ and let $\alpha, \beta \in \N^n$ such that $|\alpha|/m+|\beta|/k =p+1.$ It is easy to verify that we can find $\gamma, \delta \in \N^n,$ with $\gamma \leq \alpha, \delta \leq \beta$ such that $|\gamma|/m+|\delta|/k =p$ and $|\alpha-\gamma|/m+|\beta-\delta|/k =1.$ Then by \eqref{ellipticestimate} we can write
\begin{eqnarray*}
\| x^{\beta}D^\alpha u \|_{L^2} &\leq & \| x^{\beta-\delta}D^{\alpha-\gamma}(x^\delta D^\gamma u) \|_{L^2} + \| x^{\beta-\delta}[x^\delta, D^{\alpha-\gamma}]D^\gamma u \|_{L^2} \\
& \leq & C\| P(x^\delta D^\gamma u) \|_{L^2} + \| x^{\beta-\delta}[ x^\delta, D^{\alpha-\gamma}]D^\gamma u\|_{L^2} \\
& \leq & I_1 + I_2 +I_3,
\end{eqnarray*}
where 
$$I_1 = C \|x^\delta D^\gamma (Pu) \|_{L^2}, \qquad I_2 =C \| [P, x^\delta D^\gamma ] u \|_{L^2}, \qquad I_3 = \| x^{\beta-\delta}[ x^\delta, D^{\alpha-\gamma}]D^\gamma u\|_{L^2}.$$
Let now 
$$J_h = \sum_{\frac{|\alpha|}{m}+\frac{|\beta|}{k} =p+1}I_h, \qquad Y_h = \lambda^{-p-1} (p+1)!^{-\frac{kmt}{k+m}}J_h, \quad h=1,2,3.$$
Then, obviously we have
$$|u|_{p+1} \leq J_1+J_2+J_3, \qquad  \sigma_{p+1} (\lambda, u)  \leq Y_1+Y_2+Y_3.$$
Now, since $J_1 \leq C_1|Pu|_p$ for some $C_1 >0$, then we have $Y_1 \leq (p+1)^{-\frac{kmt}{k+m}} \sigma_p(\lambda, Pu),$ if $\lambda \geq C_1^{-1}$. To estimate $J_2$ and $Y_2$ we observe that
$$[P, x^\delta D^\gamma ]u = \sum_{\frac{|\tilde{\alpha}|}{m}+\frac{|\tilde{\beta}|}{k} \leq 1} c_{\tilde{\alpha} \tilde{\beta}}
[x^{\tilde{\beta}}D^{\tilde{\alpha}}, x^\delta D^\gamma] u,
$$
and that
$$
[x^{\tilde{\beta}}D^{\tilde{\alpha}}, x^\delta D^\gamma] u = \hskip-3pt \sum_{0 \neq \tau \leq \tilde{\alpha}, \tau \leq \delta} \hskip-4pt C_{\tilde{\alpha} \delta \tau}x^{\delta+\tilde{\beta}-\tau}D^{\gamma + \tilde{\alpha}-\tau} u - \hskip-3pt \sum_{0 \neq \tau \leq \tilde{\beta}, \tau \leq \gamma} \hskip-4pt C_{\tilde{\beta} \gamma \tau}x^{\delta+\tilde{\beta}-\tau}D^{\gamma + \tilde{\alpha}-\tau} u.
$$
where the constants $|C_{\tilde{\alpha} \delta \tau}|$ and $| C_{\tilde{\beta} \gamma \tau}|$ can be estimated by
$C_2 \, p^{|\tau|}$ for some positive constant $C_2$ independent of $p$.
We observe now that in both the sums above we have $$r= \frac{|\gamma + \tilde{\alpha}-\tau|}{m} + \frac{|\delta+\tilde{\beta}-\tau|}{k} = p+ \frac{| \tilde{\alpha}|}{m} + \frac{ |\tilde{\beta}|}{k} - \frac{m+k}{km} |\tau| \leq p+1- \frac{m+k}{km} |\tau|,$$
hence in particular we have $ 0 \leq r <p+1
$ since $|\tau|>0$. Moreover, we have
$$|\tau| \leq \frac{km}{m+k}(p+1-r).$$ In view of these considerations, we easily obtain
$$J_2 \leq C_3 ( J'_2 + p^{\frac{km}{k+m} }|u|_p + J''_2),$$
where 
$$J'_2 = \sum_{p<r<p+1} p^{\frac{km}{k+m}(p+1-r)}|u|_r,$$
$$J''_2 =\sum_{0 \leq r < p}p^{\frac{km}{k+m}(p+1-r)}|u|_r.$$
Now, applying Lemma \ref{AntiWick} to $J'_2$ with $$\varepsilon =(4C_3)^{-1} p^{-\frac{km}{k+m}(p+1-r)},$$ and using standard factorial inequalities we obtain
$$J'_2 \leq (4C_3)^{-1}|u|_{p+1} + C_4 p^{\frac{km}{k+m}}|u|_p + C_5^{p+1} (p+1)!^{\frac{km}{k+m}}|u|_0.$$
Similarly, writing $$J''_2 = p^{\frac{km}{k+m}(p+1)}|u|_0 + \sum_{q=0}^{p-1}\sum_{q < r < q+1}p^{\frac{km}{k+m}(p+1-r)}|u|_r$$ and applying Lemma \ref{AntiWick} to each term of the sum above with
$$\varepsilon = p^{-\frac{km}{k+m}(q+1-r)},$$
we get
\begin{eqnarray*}
J''_2 &\leq& C_6^{p+1} (p+1)!^{\frac{km}{k+m}}|u|_0 + C_7\sum_{q=0}^{p-1}\left[p^{\frac{km}{k+m}(p-q)}|u|_{q+1} + 
p^{\frac{km}{k+m}(p-q+1)}|u|_q \right] 
\\ &\leq& C_8^{p+1} (p+1)!^{\frac{km}{k+m}}|u|_0 + C_9 \sum_{q=1}^{p} p^{\frac{km}{k+m}(p-q+1)}|u|_q,
\end{eqnarray*}
from which we get
$$J_2 \leq \frac14 |u|_{p+1} + \tilde{C}^{p+1} (p+1)!^{\frac{km}{k+m}}|u|_0 + C' \sum_{q=1}^{p} p^{\frac{km}{k+m}(p-q+1)}|u|_q$$
for some positive constants $C', \tilde{C}$ independent of $p$.
From the estimates above, taking $\lambda$ sufficiently large and using the fact that $t \geq 1$, we obtain
$$Y_2 = \lambda^{-p-1}(p+1)!^{-\frac{kmt}{k+m}}J_2 \leq \frac14  \sum_{h=0}^{p+1} \sigma_h( \lambda, u).$$
Analogous estimates can be derived for $Y_3$ and yield \eqref{stimasigma}. We leave the details for the reader.
\end{proof}

Starting from \eqref{stimasigma} and arguing by induction on $p$ it is easy to prove the following result. We omit the proof for the sake of brevity.

\begin{lemma} \label{inductiveest}
For every $p \in \N, t \geq 1$ and $\lambda >0$ sufficiently large we have
$$\sigma_p(u, \lambda) \leq 2^{p} \sigma_0(u, \lambda) + \sum_{\ell=1}^p
2^{p-\ell} \binom{p}{\ell} (\ell !)^{-\frac{kmt}{k+m}}\sigma_0(P^\ell u, \lambda).$$ \end{lemma}

\textit{Proof of Theorem \ref{t1}.} The fact that the Gelfand-Shilov regularity of $u$ implies \eqref{iteratesestimate} is easy to prove and we omit the details. In the opposite direction, by Proposition \ref{char} it is sufficient to prove that $u$ satisfies \eqref{charest} for every $r>0$. From the previous estimate, we have, for every $p \in \N$:
$$\sigma_p(u, \lambda) \leq C + \sum_{\ell=1}^p
2^{p-\ell} \binom{p}{\ell} C^{\ell+1} \leq C(2+C)^{p+1}.$$
Therefore
$$|u|_p  \leq C^{p+1}p!^{\frac{kmt}{k+m}}$$
for a new constant $C>0$, which gives \eqref{charest} in the case $r \in \N$. If $ r>0$ is not integer, then $p<r<p+1$ for some $ p \in \N$ and we can apply Lemma \ref{AntiWick} which yields
\begin{eqnarray*}
|u|_r &\leq& \varepsilon |u|_{p+1}+C \varepsilon^{-\frac{r-p}{p+1-r}}|u|_p + C^p (p!)^{\frac{km}{k+m}}|u|_0 \\
&\leq& \varepsilon C_1^{p+1}(p+1)!^{\frac{kmt}{k+m}} + C_1^p \varepsilon^{-\frac{r-p}{p+1-r}} (p+1)!^{\frac{kmt}{k+m}}
 + C_1^p (p+1)!^{\frac{kmt}{k+m}} \leq C_2^{r+1} r^{\frac{kmrt}{k+m}}.
\end{eqnarray*}
Then, by Proposition \ref{char} we conclude that $u \in \mathcal{S}^{\frac{kt}{k+m}}_{\frac{mt}{k+m}}(\R^n)$. Similarly we argue for 
$u \in \Sigma^{\frac{kt}{k+m}}_{\frac{mt}{k+m}}(\R^n)$. \qed
\\

\textit{Proof of Theorem \ref{t2}}. The equivalence between $ii)$ and $iii)$ is obvious. Moreover $iii)$ is equivalent to $iv)$ in view of \eqref{aseigen}. The arguments are similar for $\mathcal{S}^{\frac{kt}{k+m}}_{\frac{mt}{k+m}}(\R^n)$ and $\Sigma^{\frac{kt}{k+m}}_{\frac{mt}{k+m}}(\R^n)$ classes. To conclude the proof we will show the equivalence between $i)$ and $iv)$. 
We first observe that 
$$ \|P^Mu\|^2_{L^2}=\|\sum\limits_{j=1}^\infty u_jP^M
\varphi_j \|_{L^2}^2=\sum\limits_{j=1}^\infty \lambda_j^{2M}|u_j|^2,
$$
in view of  Parseval identity. By \eqref{aseigen} it
follows that
\begin{eqnarray}\label{eq3.2}
C_1\|P^Mu\|^2_{L^2}\leq \sum\limits_{j=1}^\infty j^{2Mkm/(n(k+m))}|u_j|^2
\leq C_2\|P^Mu\|^2_{L^2}
\end{eqnarray}
for suitable positive constants $C_1,C_2$. Now if $iv)$ holds, then we have
$$ |u_j|^2\leq e^{-\epsilon  j^{1/(nt)}}$$
for some new constant $\epsilon>0$. Then from the first estimate
in (\ref{eq3.2}) we have for some $C>0$
\begin{eqnarray}
\|P^Mu\|^2_{L^2} &\leq & C\sum\limits_{j=1}^\infty
j^{2Mkm/(n(m+k))}e^{-\epsilon j^{1/(nt)}}
 \\ &\leq&
\tilde{C}\sup_{j \in \N}j^{2Mmk/(n(m+k))}e^{-\epsilon j^{1/(nt)}}
\label{eq3.3} \end{eqnarray}
with
$$\tilde{C}=C\sum\limits_{j=1}^\infty e^{-\epsilon j^{1/(nt)}}.$$
Moreover, for any fixed $\omega >0$ we have
$$e^{\omega j^{1/(nt)}}=\sum\limits_{M=0}^\infty \frac{\omega^Mj^{M/(nt)}}{M!}.$$
This implies that for every $M \in {\mathbb N}$:
\begin{eqnarray}\label{eq3.4}
j^{M/(nt)}e^{-\omega j^{1/(nt)}}\leq \omega^{-M}M!
\end{eqnarray}
Taking the $2kmt/(k+m)$-th power of both sides of (\ref{eq3.4}) and
applying in the last estimate in (\ref{eq3.3}) with $$\omega =2\epsilon kmt/(k+m),$$ we obtain 
\begin{eqnarray*}
\|P^Mu\|^2_{L^2}\leq \tilde{C}\omega^{-\frac{2Mkmt}{k+m}}(M!)^{\frac{2mkt}{m+k}},
\end{eqnarray*}
which gives $i)$ in view of Theorem \ref{t1}. \\
$i) \Rightarrow ii)$ Viceversa assume that $u \in \mathcal{S}^{\frac{kt}{k+m}}_{\frac{mt}{k+m}}(\R^n)$.
In view of $iv)$ it is sufficient to show that
\begin{equation}\label{bord}\sup_{j \in \N} |u_j|^2 e^{\epsilon j^{\frac1{nt}}}< +\infty.
\end{equation}
Theorem \ref{t1} and the second inequality in \eqref{eq3.2} imply that
$$\frac{j^{\frac{2Mkm}{n(k+m)}}}{C^M (M!)^{\frac{2kmt}{k+m}}}|u_j|^2 \leq C$$
for every $j, M \in \N$ and for some $C$ independent of $j$ and $M$.
Taking the supremum of the left-hand side over $M$ we get \eqref{bord} with $\epsilon =\frac{2kmt}{k+m}C^{-\frac{k+m}{2kmt}}.$ This concludes the proof. \qed

\begin{section}{Generalizations}
We list some possible generalizations of the preceding results. First, one can replace the hypothesis of positivity for the operator $P$ by assuming that $P$ is normal, i.e. $P^*P = PP^*$. This guarantees the existence of an orthonormal basis of eigenfunctions $\varphi_j, j \in \N$, with eigenvalues $\lambda_j, \lim\limits_{j \to \infty}|\lambda_j| =+\infty$, see \cite{shubin}, and we may then proceed as before, cf. \cite{S2}. \\ Another possible generalization consists in replacing $L^2$ norms with $L^p$ norms, $1<p<\infty$. Let us observe that the basic estimate \eqref{ellipticestimate} is valid also for $L^p$ norms, see \cite{GM, M}, and it seems easy to extend Theorem \ref{t1} in this direction.  \\ A much more challenging problem is an analogous characterization of the classes $\mathcal{S}^\mu_\nu( \R^n)$ when $\kappa = \mu/\nu = k/m $ is irrational. First difficulty, in this case, is given by an appropriate choice of the operator $P$. In fact, the natural candidates 
$$P= (-\Delta)^{m/2} +(1+|x|^2)^{k/2}, \qquad m \in 2\N, k >0, k \notin 2\N$$
can be easily treated in the setting of temperate distributions but results of Gelfand-Shilov regularity, extending those in \cite{CGR2}, are missing for them.

\end{section}
 \vskip0.3cm

\textbf{Note.} With great sorrow, Marco Cappiello, Stevan Pilipovic and Luigi Rodino inform that their friend Todor Gramchev passed away on October 17, 2015. He inspired and collaborated to the initial version of the present paper and appears here as co-author.

\bibliographystyle{amsplain}

\end{document}